\documentclass[12pt]{article}
\usepackage{fullpage, amsmath, amsthm}
\usepackage{url}
\usepackage{hyperref}

\newtheorem{theorem}{Theorem}

\newtheorem{lemma}[theorem]{Lemma}

\theoremstyle{remark}

\theoremstyle{definition}

\newcommand{\seqnum}[1]{\href{http://oeis.org/#1}{\underline{#1}}}

\title{Improved estimates for the number of privileged words}

\author{Jeremy Nicholson and Narad Rampersad \thanks{The first author
    is supported by an NSERC USRA, the second by an NSERC Discovery
    Grant.} \\
Department of Mathematics and Statistics \\
University of Winnipeg \\
515 Portage Avenue \\
Winnipeg, Manitoba R3B 2E9 (Canada)\\
\texttt{n.rampersad@uwinnipeg.ca, jnich998@hotmail.com}}

\begin{document}

\maketitle

\begin{abstract}
In combinatorics on words, a word $w$ of length $n$ over an alphabet of 
size $q$ is said to be \emph{privileged} if $n \leq 1$ or if $n \geq 2$ and $w$ has a 
privileged border that occurs exactly twice in $w$. Forsyth, 
Jayakumar and Shallit proved that there exist at least 
$2^{n-5}/n^2$ privileged binary words of length $n$.
Using the work of Guibas and Odlyzko, we prove that there are
constants $c$ and $n_0$ such that for 
$n\geq n_0$, there are at least $\frac{cq^n}{n(\log_q n)^2}$
privileged words of length $n$ over an alphabet of size $q$.
Thus, for $n$ sufficiently large,
we improve the earlier bound set by Forsyth,
Jayakumar and Shallit and generalize for all $q$.
\end{abstract}

\maketitle

\section{Introduction}
The class of \emph{privileged words} was recently introduced by
Kellendonk, Lenz, and Savinien \cite{KLS13} and studied further by
Peltom\"aki \cite{Pel13}.  This class of words has been studied due to its
connection with palindromes and the class of so-called ``rich
words''.  A word $w$ is \emph{privileged} if either
\begin{itemize}
\item $w$ is a single letter, or
\item $w$ has a privileged border that occurs exactly twice in $w$.
\end{itemize}
Recall that a \emph{border} of a word $w$ is a non-empty word that is
both a prefix and a suffix of $w$.

The first few privileged words (up to length $4$) over the alphabet
$\{0,1\}$ are
\[
\epsilon, 0, 1, 00, 11, 000, 010, 101, 111, 0000, 0110, 1001, 1111.
\]

Forsyth, Jayakumar, Peltom\"aki, and Shallit \cite{FJS13} looked at the problem of
enumerating privileged words over a binary alphabet.  The number of
such words of length $n$ is given by sequence \seqnum{A231208} of the Online
Encyclopedia of Integer Sequences.  This sequence begins
\[
1, 2, 2, 4, 4, 8, 8, 16, 20, 40, 60, 108, \ldots
\]
Forsyth et al.\ proved that
there exist at least $2^{n-5}/n^2$ privileged binary words of length
$n$.  In their paper they sketch a method for potentially improving
this estimate.  In the present paper we apply some results of Guibas
and Odlyzko \cite{GO78} on the size of prefix-synchronized codes to
carry out this method.  We are thus able to obtain the following
asymptotic improvement to the result of Forsyth, Jayakumar, and
Shallit, and also generalize the result to arbitrary alphabets.

\begin{theorem}\label{main}
Let $q \geq 2$.
There exist constants $c$ and $n_0$ such that for $n \geq n_0$, there are at least
$\frac{cq^n}{n(\log_q n)^2}$
privileged words of length $n$ over an alphabet of size $q$.
\end{theorem}

The estimates given in this theorem derive from the asymptotic
analysis of maximal prefix-synchronized codes carried out by Guibas
and Odlyzko \cite{GO78}.  Given an alphabet of size $q$, a block
length $N$ and a prefix $P$ of length $p<N$, a
\emph{prefix-synchronized code} is a set of length-$N$ codewords with the
property that every codeword starts with a fixed prefix $P=a_1a_2
\cdots a_p$, and furthermore, for any codeword $a_1a_2 \cdots
a_pb_1b_2 \cdots b_{N-p}$, the prefix $P$ does not appear as a
factor of $a_2 \cdots a_pb_1 \cdots b_{N-p}a_1 \cdots a_{p-1}$.  In
other words, the border $a_1a_2 \cdots a_p$ of $a_1a_2 \cdots a_pb_1
\cdots b_{N-p} a_1a_2 \cdots a_p$ occurs exactly twice in this word.

Next, we define $G(N)=G_P(N)$ as the size of 
a maximal prefix-synchronized code with these parameters. In other words, $G_P(N)$ is the number of $q$-ary words
$a_1 \cdots a_{N+p}$ such that $a_k \cdots a_{k+p-1}=P$ for $k=1$ and
$k=N+1$, and for no other $k$ with $1 \leq k \leq N+1$. If we take $n=N+p$ where the prefix $P$ of length $p$ is a privileged word,
then $G_P(N)$ counts all words of length $n$ with a privileged border $P$ that occurs exactly twice in $w$. All these words are necessarily privileged words.
If we sum up $G_P(N)$ over all privileged $P$ of length $p$ for some $p$,
then we obtain a lower bound for the number of privileged words of
length $n$.

\section{Proof of Theorem~\ref{main}}

 To prove Theorem~\ref{main} we begin with some lemmas. In all the 
following lemmas and for the rest of this document, $q$ is the size of 
the alphabet and it will be a fixed integer $\geq2$, $p$ is the size of
the prefix $P$, and $Q=y_1y_2 \cdots y_p$ is the \emph{autocorrelation} of $P$,
defined as follows.  If $P = a_1\cdots a_p$, then for $i=1,\ldots,p$ we define
\[
y_i = \begin{cases} 1 & \text{if } a_{i+1}\cdots a_k = a_1\cdots
  a_{k-i}, \\ 0 & \text{otherwise.}\end{cases}
\]
We also define the polynomial
\[
f(z)=f_Q(z)=\sum_{i=1}^{p} y_iz^{p-i}.
\]

The next series of lemmas are Lemmas~3--6 of \cite{GO78}.

\begin{lemma}\label{rho_def}
If $p$ is sufficiently large, then $1+(z-q)f(z)$ has exactly one zero
$\rho$ that satisfies $|\rho| \geq 1.7$.
\end{lemma}

Since there is only one such root $\rho$, it follows that this $\rho$ is real.
In what follows, the quantity $\rho$ is the $\rho$ specified by the
previous lemma.

\begin{lemma}\label{rho_bounds}
If $p$ is sufficiently large, then
\[
\ln \rho=\ln q-\frac{1}{qf(q)}-\frac{f'(q)}{qf(q)^3}-\frac{1}{2q^2f(q)^2}+O\left(\frac{p^2}{q^{3p}}\right).
\]
\end{lemma}

Define $R_Q$ by
\[
R_Q\rho = \frac{(q-\rho)^2\rho^{p-1}}{1-(q-\rho)^2f'(\rho)}.
\]

\begin{lemma}\label{G_formula}
$G(N)=R_Q\rho^N + O((1.7)^N)$
\end{lemma}

\begin{lemma}\label{R_Q_bounds}
If $p$ is sufficiently large, then
\[
\ln R_Q = (p-2)\ln q -2\ln(f(q))+\frac{3f'(q)}{f(q)^2}-\frac{p-2}{qf(q)}+O\left(\frac{p^2}{q^{2p}}\right).
\]
\end{lemma}

In what follows, $c$'s, $d$'s and Greek letters denote positive
constants.

\begin{lemma}\label{G_lower_bound}
Let $p$ be the unique integer such that
\[
\frac{\ln q}{q-1}q^p \leq N<\frac{\ln q}{q-1}q^{p+1}.
\]
Let $P$ be a prefix of length $p$ and let $n=N+p$.  There exist
constants $N_0$ and $d$ such that for $N>N_0$ we have
$$G_P(N) \geq dq^n/n^2.$$  (The constant $d$ may depend on $q$ but not
on $N$.)
\end{lemma}

\begin{proof}
If $p=\lfloor\log_q N +\log_q(q-1)-\log_q(\ln q)\rfloor$, then
$G(N)=R_Q\rho^N + O((1.7)^N)$ by applying Lemma~\ref{G_formula}.  By
Lemmas~\ref{rho_bounds} and \ref{R_Q_bounds} we have
\begin{eqnarray*}
\ln(R_Q\rho^N)&=&\ln R_Q + N\ln \rho\\ 
&=&(p-2)\ln q -2\ln(f(q))+\frac{3f'(q)}{f(q)^2}-\frac{p-2}{qf(q)}+O\left(\frac{p^2}{q^{2p}}\right)\\
&+&N\ln q-\frac{N}{qf(q)}-\frac{Nf'(q)}{qf(q)^3}-\frac{N}{2q^2f(q)^2}+O\left(\frac{Np^2}{q^{3p}}\right).
\end{eqnarray*}
Therefore,
\begin{eqnarray*}
G(N)&=&R_Q\rho^N + O((1.7)^N)=\exp(\ln(R_Q\rho^N)) + O((1.7)^N)\\
&=&\frac{q^{N+p-2}}{f(q)^2}\exp\left(\frac{3f'(q)}{f(q)^2}-\frac{p-2}{qf(q)}-\frac{N}{qf(q)}-\frac{Nf'(q)}{qf(q)^3}-\frac{N}{2q^2f(q)^2}+O\left(\frac{Np^2}{q^{3p}}+\frac{p^2}{q^{2p}}\right)\right)\\
&\phantom{=}&\hspace{2ex} \mathop{+} O((1.7)^N).
\end{eqnarray*}

Since the first digit of $Q$ will always be a 1, $f(q)$ will have the leading term $q^{p-1}$.  Let $\alpha, \beta, \gamma, \delta$, and $N_0$ be positive constants such that the following inequalities are valid for all $N \geq N_0$.

$$0\leq\frac{3f'(q)}{f(q)^2} \leq3 \frac{(p-1)^2q^{p-2}}{q^{2p-2}}\leq\frac{3(p-1)^2}{q^p}\leq 2$$

$$\frac{p-2}{qf(q)}\leq\frac{p-2}{q^p}\leq\frac{p}{q^p}\leq 1/q$$

$$\frac{N}{qf(q)}\leq \frac{N}{q^p}\leq\frac{N}{q^{\lfloor{\log_q N
      -\log_q(\ln q)}\rfloor}}\leq \alpha\frac{N}{q^{\log_q N
    -\log_q(\ln q)}}\leq \alpha\frac{N\ln q}{N} \leq \beta$$

$$\frac{Nf'(q)}{qf(q)^3}\leq \frac{N(p-1)^2q^{p-2}}{q^{3p-2}}\leq \frac{N(p-1)^2}{q^{2p}}\leq \frac{N}{q^p}\frac{(p-1)^2}{q^p}\leq 2\beta$$

$$\frac{N}{2q^2f(q)^2}\leq\frac{N}{2q^{2p}}\leq\frac{N}{2q^{2\lfloor{\log_q
      N -\log_q(\ln q)}\rfloor}}\leq\frac{\delta N}{q^{2\log_q N }}\leq\frac{\delta N}{N^2} \leq\delta$$

$$\left| O\left(\frac{Np^2}{q^{3p}}+\frac{p^2}{q^{2p}}\right) \right|
= \left| O\left(\frac{p^2}{q^{2p}}\right) \right| \leq \gamma.$$

Thus 
\begin{eqnarray*}
G_{P}(N)&=&\frac{q^{N+p-2}}{f(q)^2}\exp\left(\frac{3f'(q)}{f(q)^2}-\frac{p-2}{qf(q)}-\frac{N}{qf(q)}-\frac{Nf'(q)}{qf(q)^3}-\frac{N}{2q^2f(q)^2}+O\left(\frac{Np^2}{q^{3p}}+\frac{p^2}{q^{2p}}\right)\right)\\
&\phantom{=}&\hspace{2ex} \mathop{+} O((1.7)^N)\\
&\geq& \frac{d_1q^{N+p-2}}{f(q)^2}\geq\frac{d_1q^{N+p-2}}{((1-q^p)/(1-q))^2}\geq \frac{d_2q^{N+p-2}}{(1-q^p)^2}
\geq\frac{d_2q^{N+p-2}}{q^{2p}-2q^p+1}\\
&\geq&\frac{d_2q^{N+p-2}}{q^{2p}+1}\geq\frac{d_2q^{N+p-2}}{2q^{2p}}
\geq\frac{d_3q^N}{q^p}\geq\frac{d_3q^N}{q^{\lfloor{\log_q N +\log_q(q-1)-\log_q(\ln q)}\rfloor}}\\
&\geq& \frac{d_3q^N}{q^{\log_q N +\log_q(q-1)-\log_q(\ln q)}}\geq\frac{d_3q^N\ln q}{N(q-1)}
\geq\frac{d_4q^N}{N}\geq\frac{d_4q^{n-p}}{n}\\
&\geq&\frac{d_4q^{n-(\log_q N +\log_q(q-1)-\log_q(\ln q))}}{n}\geq\frac{d_4q^n\ln q}{nN(q-1)}
\geq\frac{dq^n}{n^2}.
\end{eqnarray*}
\end{proof}

We can now complete the proof of Theorem~\ref{main}.

\begin{proof}[Proof of Theorem~\ref{main}]
We define the function $B(n,q)$ as the number of privileged words of
length $n$ over an alphabet of size $q \geq 2$. Let $n=N+p$ where
$p=\lfloor{\log_q N +\log_q(q-1)-\log_q(\ln q)}\rfloor$.  Let $n_0$ be
a constant such that whenever $n \geq n_0$ we have $p \geq N_0$, where
$N_0$ is the constant mentioned in Lemma~\ref{G_lower_bound}.  Then
for $n \geq n_0$, we have
\begin{eqnarray*}B(n,q)&\geq& \sum_{\substack{P \text{ privileged}\\ |P| = p}} G_P(N)\\
&\geq& \left(\frac{dq^n}{n^2}\right)B(\lfloor{\log_q N +\log_q(q-1)-\log_q(\ln q)}\rfloor,q)\\
&\geq&\left(\frac{dq^n}{n^2}\right)\left(\frac{c_1q^{\lfloor{\log_q N +\log_q(q-1)-\log_q(\ln q)}\rfloor}}{(\lfloor{\log_q N +\log_q(q-1)-\log_q(\ln q)}\rfloor)^2}\right)\\
&\geq&\left(\frac{c_2q^n}{n^2}\right)\left(\frac{q^{\log_q N +\log_q(q-1)-
\log_q(\ln q)}}{(\log_q N +\log_q(q-1)-\log_q(\ln q))^2}\right)\\
&\geq&\left(\frac{c_2q^n}{n^2}\right)\left(\frac{N(q-1)}{(\ln q)(1+\log_q N)^2}\right)\\
&\geq&\left(\frac{c_3q^n}{n^2}\right)\left(\frac{N}{(\log_q q+\log_q N)^2}\right)\\
&\geq&\left(\frac{c_3q^n}{n^2}\right)\left(\frac{N}{(\log_q qN)^2}\right)\\
&\geq&\left(\frac{c_3q^n}{n^2}\right)\left(\frac{n}{(\log_q qn)^2}-\frac{p}{(\log_q qn)^2}\right)\\
&\geq&\left(\frac{c_3q^n}{n^2}\right)\left(\frac{n}{(\log_q qn)^2}-\frac{\log_q n}{(\log_q qn)^2}-\frac{1}{(\log_q qn)^2}\right)\\
&\geq&\left(\frac{c_3q^n}{n(\log_q n)^2}\right)\left(\frac{(\log_q n)^2}{(\log_q qn)^2}-\frac{(\log_q n)^3}{n(\log_q qn)^2}-\frac{(\log_q n)^2}{n(\log_q qn)^2}\right)\\
&\geq& \frac{cq^n}{n(\log_q n)^2},
\end{eqnarray*}
 since $$\frac{(\log_q n)^2}{(\log_q qn)^2}-\frac{(\log_q n)^3}{n(\log_q qn)^2}-\frac{(\log_q n)^2}{n(\log_q qn)^2}$$ is positive and increases for $n>2$.
This completes the proof.
\end{proof}

\bigskip
\hrule
\bigskip
\noindent 2010 {\it Mathematics Subject Classification}: Primary
68R15.

\noindent \emph{Keywords}: privileged word.

\begin{thebibliography}{9}
\bibitem{FJS13}
M. Forsyth, A. Jayakumar, J. Peltom\"aki, and J. Shallit, ``Remarks on privileged
words'', Int. J. Found. Comput. Sci. 27 (2016), 431--442.

\bibitem{GO78}
L. Guibas and A. Odlyzko, ``Maximal prefix-synchronized codes'', SIAM
J. Appl. Math. 35 (1978), 401--418.

\bibitem{KLS13}
J. Kellendonk, D. Lenz, and J. Savinien, ``A characterization of subshifts
with bounded powers'', Discrete Math. 313 (2013), 2881--2894.

\bibitem{Pel13}
J. Peltom\"aki, ``Introducing privileged words:  Privileged complexity
of Sturmian words'', Theoret. Comput. Sci. 500 (2013), 57--67.
\end{thebibliography}
\end{document}